\documentclass[11pt,a4paper]{article}
\usepackage{amsmath,amsfonts,amssymb,mathrsfs}
\usepackage{graphicx}
\usepackage{subfigure}
\usepackage{pstricks}
\usepackage{pst-node}
\usepackage{xcolor}
\usepackage{mathscinet}
\usepackage{ntheorem}
\usepackage{cite}

\newrgbcolor{blu}{.1 0 .7}
\newrgbcolor{lgrey}{.9 .9 .9}
\newrgbcolor{grey}{.7 .7 .7}
\newrgbcolor{dgrey}{.3 .3 .3}
\newrgbcolor{lred}{1 .8 .8}
\newrgbcolor{mlred}{1 .7 .7}
\newrgbcolor{bblue}{.6 .6 1}
\newrgbcolor{lblue}{.8 .8 1}
\newrgbcolor{mblue}{.2 .25 1}
\newrgbcolor{dblue}{.3 .32 1}
\newrgbcolor{ddblue}{.4 .4 1}
\newrgbcolor{lgreen}{.8 1 .8}
\newrgbcolor{mgreen}{ .7 1 .7}
\newrgbcolor{ggreen}{ .1 .75 .35}

\newtheorem{theorem}{\bf Theorem}[section]
\newtheorem{lemma}[theorem]{\bf Lemma}
\newtheorem{corollary}[theorem]{\bf Corollary}
\newtheorem{definition}[theorem]{\bf Definition}
\newtheorem{remark}[theorem]{\bf Remark}
\newtheorem{proposition}[theorem]{\bf Proposition}

\newenvironment{proof}{\noindent {\sc Proof.}}{\hfill $\square$}

\newenvironment{marev}{\color{red}}{\color{black}}
\newcommand{\bmicr}{\begin{marev}}
\newcommand{\emicr}{\end{marev}}

\newenvironment{sergiorev}{\color{blue}}{\color{black}}
\newcommand{\bsr}{\begin{sergiorev}}
\newcommand{\esr}{\end{sergiorev}}

\definecolor{ao(english)}{rgb}{0.0, 0.5, 0.0}

\newenvironment{frarev}{\color{ao(english)}}{\color{black}}
\newcommand{\bfra}{\begin{frarev}}
\newcommand{\efra}{\end{frarev}}

\newcommand{\denoterow}[1]{\rlap{\hspace{1em}$\leftarrow{}$ $j^{th}$}}

\newcommand{\R}{\mathbb{R}}

\newcommand{\N}{\mathbb{N}}

\def \Z {{\mathbb{Z}}}

\def \K  {K}  
\def \L {\mathscr{L}}
\def \K {\mathcal{K}}

\def \Q {\mathcal{Q}}
\def \phi {\varphi}

\def \div {{\text{\rm div}} \hspace{.5mm}}
\def \loc {{\text{\rm loc}}}
\def \p{\partial}

\def \diag {{\text{\rm diag}}}
\def \trace {{\text{\rm tr}}}
\def \det {{\text{\rm det }}}
\def \meas {{\text{\rm meas}}}

\def \D {{\Delta}}

\def \a {{\alpha}}
\def \b {{\beta}}
\def \cc {{\chi}}
\def \g {{\gamma}}
\def \d {{\delta}}
\def \e {{\varepsilon}}

\def \epsilon {{\varepsilon}}

\def \k {{\kappa}}
\def \l {{\lambda}}
\def \r {{\rho}}

\def \t {{\tau}}
\def \m {{\mu}}

\def \x {{\xi}}

\def \z {{\zeta}}

\def \G {{\Gamma}}
\def \O {{\Omega}}

\def \OO {{\mathbb{O}}}
\def \I {{\mathbb{I}}}

\usepackage[english]{babel}
 
\begin{document}
\title{Moser's estimates for degenerate Kolmogorov equations with non-negative divergence lower order coefficients}
\date{}
\author{
\sc{Francesca Anceschi}
\thanks{Dipartimento di Scienze Fisiche, Informatiche e Matematiche, 
Universit\`{a} di Modena e Reggio Emilia, Via Campi 213/b, 41125 Modena (Italy).
E-mail: francesca.anceschi@unimore.it} \hspace{8mm}
\sc{Sergio Polidoro}
\thanks{Dipartimento di Scienze Fisiche, Informatiche e Matematiche, 
Universit\`{a} di Modena e Reggio Emilia, Via
Campi 213/b, 41125 Modena (Ialy). E-mail: sergio.polidoro@unimore.it} \\
{\sc{Maria Alessandra Ragusa} 
\thanks{Dipartimento di Matematica e Informatica, Università degli Studi di  
Catania, Viale Andrea Doria, 5, 95125 Catania (Italy), RUDN  
University, 6 Miklukho - Maklay St, Moscow, 117198 (Russia). E-mail: 
maragusa@dmi.unict.it} }}

\maketitle

\bigskip

\begin{abstract}
We prove $L^\infty_\loc$ estimates for positive solutions to the following degenerate second order partial differential 
equation of Kolmogorov type with measurable coefficients of the form
\begin{align*}
	\sum \limits_{i,j=1}^{m_0}  \p_{x_i} \left( a_{ij}(x,t) \p_{x_j} u(x,t) \right) 
     &+ \sum \limits_{i,j=1}^{N} b_{ij} x_j \p_{x_i} u(x,t) - 
    \p_t u(x,t)  + \\
    &+ \sum \limits_{i=1}^{m_0} b_i(x,t) \p_i u(x,t)  - 
     \sum \limits_{i=1}^{m_0} \p_{x_i} \left( a_i(x,t) u(x,t) \right) + c(x,t) u(x,t) = 0
\end{align*}
where $(x,t)= ( x_1, \ldots, x_N, t) = z$ is a point of 
$\R^{N+1}$, and $1 \le m_0 \le N$. $(a_{ij})$ is an uniformly positive symmetric matrix with bounded measurable 
coefficients, $(b_{ij})$ is a constant matrix. We apply the Moser's iteration method to prove the local boundedness of 
the solution $u$ under minimal integrability assumption on the coefficients.
\end{abstract} 

\normalsize\

\section{Introduction}
We consider second order partial differential operators of Kolmogorov-Fokker-Planck 
type of the form
\begin{equation} \label{eq1}
\begin{split}
      \L u(x,t) := & 
      \sum \limits_{i,j=1}^{m_0} \p_{x_i} \left( a_{ij}(x,t) \p_{x_j} u(x,t) \right) 
      + \sum \limits_{i,j=1}^{N} b_{ij} x_j \p_{x_i} u(x,t) - \p_t u(x,t)  + \\
     + & \sum \limits_{i=1}^{m_0} b_i(x,t) \p_i u(x,t) -
     \sum \limits_{i=1}^{m_0} \p_{x_j} \left( a_i(x,t) u(x,t) \right) + c(x,t) u(x,t) = 0,
\end{split}
\end{equation}
in some open set $\Omega \subseteq \R^{N+1}$. Here $z = (x,t) = ( x_1, \ldots, x_N, t)$ denotes a point of $\R^{N+1}$, 
and $1 \le m_0 \le N$. In the sequel we will use the following notation
\begin{equation*}
    A(x,t) = \left( a_{ij}(x,t) \right)_{1 \le i,j \le N}, 
 \end{equation*}
where $ a_{ij}$ is the coefficient appearing in \eqref{eq1} for $i, j=1, \dots, m_0$, while $a_{ij} \equiv 
0$ whenever $i > m_0$ or $j > m_0$. Eventually,
\begin{align}
    \label{eq3}
      a(x,t) = \left(a_1(x,t), \ldots, a_{m_0}(x,t), 0, \ldots, 0 \right),& \qquad 
      b(x,t) = \left(b_1(x,t), \ldots, b_{m_0}(x,t), 0, \ldots, 0 \right)    
       \nonumber \\
     Y = \sum \limits_{i,j=1}^{N} &b_{ij} x_j \p_{x_i} - \p_t.
\end{align}
Then the operator $\L$ takes the following compact form
\begin{equation*}
    \L u = \div (A D u) + Y u + \langle b, D u \rangle - \div ( a u ) + c u. 
\end{equation*}
Here and in the sequel
\begin{equation} \label{notation1}
    D = (\p_{x_1}, \ldots, \p_{x_N}), \qquad \langle \cdot, \cdot \rangle, \qquad \div ,
\end{equation}
denote the gradient, the inner product, and and the divergence in $\R^N$, respectively. 
As the operator $\L$ is non degenerate with respect to the first $m_0$ components of $x$, we also introduce the 
notation
\begin{equation*} \label{notation2}
    D_{m_{0}} = (\p_{x_1}, \ldots, \p_{x_{m_{0}}}).
\end{equation*} 

We assume the following structural condition on $\L$.
\begin{description}
  \item [(H1)] The matrix $\left(a_{ij}(x,t) \right)_{i,j=1, \dots, m_0}$ is symmetric with real measurable entries. 
Moreover, $a_{ij}(x,t) = a_{ji}(x,t), \hspace{1mm} 1 \le i,j \le m_0$, and there exists a positive constant $\lambda$ 
such that 
\begin{equation*}
             \l^{-1} | \xi |^2 \le \sum \limits_{i,j=1}^{m_0} a_{ij}(x,t) \xi_i \xi_j \le \l | \xi |^2,
\end{equation*}
for every $(x,t) \in \R^{N+1}$ and $\xi \in \R^{m_0}$. The matrix $B= \left( b_{ij} \right)_{i,j=1, \ldots, N}$ is 
constant.
\end{description}

Note that the operator $\L$ is uniformly parabolic when $m_0 = N$. In this note, we are mainly interested in the 
case $m_0 < N$, that is the strongly degenerate one. It is known that the first order part of $\L$ may provide 
it with strong regularity properties. To be more specific, let's consider the operator $\K$ defined as follows:
\begin{equation} \label{eq2}
      \K u(x,t) :=  \sum \limits_{i=1}^{m_0} \p^2_{x_i} u(x,t) 
      + \sum \limits_{i,j=1}^{N} b_{ij} x_j \p_{x_i} u(x,t) - \p_t u(x,t).
\end{equation}
It is known that, if the matrix $B$ satisfies a suitable  assumption, then $\K$ is hypoelliptic. This means that, if 
$u$ is a distributional solution to $\K u = f$ in some open set $\Omega$ of $\R^{N+1}$ and $f \in C^\infty(\Omega)$, 
then $u \in C^\infty(\Omega)$ and it is a classic solution to the equation.

The hypoellipticity of $\K$ can be tested via the condition introduced by H\"ormander in \cite{H}:
\begin{equation*}
    \text{rank} \hspace{1mm} \text{Lie} ( \p_{x_1}, \ldots, \p_{x_{m_0}}, Y ) (x,t) = N+1, 
    \hspace{2mm} \forall (x,t)  \in \R^{N+1},
\end{equation*}
where $\text{Lie} ( \p_{x_1}, \ldots, \p_{x_{m_0}}, Y ) (x,t) $ denotes the Lie algebra 
generated by the first order differential operators (vector fields) $( \p_{x_1}, \ldots, \p_{x_{m_0}}, Y )$, computed 
at $(x,t)$. We refer to E. Lanconelli and one of  the  authors \cite{LP} for a characterization of the hypoellipticity 
of $\K$ in terms of the matrix $B$.

\begin{description}
  \item  [(H2)] The principal part $\K$ of $\L$ is hypoelliptic.
\end{description}

In Section \ref{preliminaries}, we recall a known structural condition on the matrix $B$ equivalent to \textbf{(H2)}. 
We remark that if $\L$ is an uniformly parabolic operator (i.e. $m_0=N$ and $B \equiv 0$), then \textbf{(H2)} is 
clearly satisfied. Indeed, the principal part of $\L$ simply is the heat operator, which is hypoelliptic and 
homogeneous 
with respect to the parabolic dilations $\d_{\l}(x,t) = (\l x, \l^2 t)$.  In the degenerate setting, $\K$ plays the 
same role that the heat operator plays in the family of the parabolic operators. For this reason, $\K$ will be 
referred to as \emph{principal part of} $\L$.

The aim of this work is to prove $L^\infty_{\loc}$ estimates for weak solutions to $\L u = 0$, by using the Moser's 
iteration method, under minimal assumptions on the integrability of the lower order coefficients $a_1, \ldots, a_{m_0}, b_1, \dots, 
b_{m_0}, c$. The Moser's iterative scheme (\cite{M1}, \cite{M2}) has been applied to degenerate parabolic operators $\L$ with no 
lower order terms by Cinti, Pascucci and one of the authors in \cite{PP} and \cite{CPP}. These results have been 
extended to operators with bounded first order coefficients by Lanconelli, Pascucci and one of the authors in 
\cite{LP} and \cite{LPP}, and to operators with first order coefficients belonging to some $L^{q}$ space by Wang and 
Zhang \cite{WZ4}. 

Our study has been inspired by the article of Nazarov and Uralt'seva \cite{NU}, who prove 
$L^\infty_{\loc}$ estimates and Harnack inequalities for uniformly elliptic and parabolic 
operators in divergence form that are those with $m_0 = N$ according to our notation.
The authors consider uniformly parabolic equations in $\R^{N+1}$
\begin{equation*}
    \L u =  \div (A Du) + \langle b, Du \rangle - \p_t u = 0,
\end{equation*}
with $b_1, \dots, b_N \in L^q (\R^{N+1})$. They prove that the Moser's iteration can be accomplished provided that 
$\frac{N+2}{2} < q \le N+2$ relying on the condition $\div  b \ge 0$ to relax the integrability assumption on $b_1, \dots, b_{m_0}$.
Here and in the sequel, the quantity $\div b$ will be understood in the distributional sense
\begin{equation*} 
  \int_{\Omega} \varphi (x,t) \div  b(x,t) dx \, dt = - 
  \int_{\Omega}  \langle b(x,t), \nabla \varphi (x,t) \rangle  dx \, dt,
\end{equation*}
for every $\varphi \in C_0^\infty(\Omega)$. Of course, also the quantity $\div a$ will be understood in the distributional sense. 
 
When considering degenerate operators, a suitable dilation group $\left( \d_r \right)_{r > 0}$ in $\R^{N+1}$ replaces 
the usual parabolic dilation $\d_{r}(x,t) = (r x, r^2 t)$, and the \emph{parabolic dimension} $N+2$ of $\R^{N+1}$ is 
replaced by a bigger integer $Q+2$, which is called \emph{homogeneous dimension}  of $\R^{N+1}$ with respect to $\left( 
\d_r \right)_{r > 0}$. Our main result will be declared in terms of this quantity, that will be introduced in Section 
\ref{preliminaries}. 

As far as it concerns degenerate operators, Wang and Zhang obtain in \cite{WZ4} the local boundedness and the H\"older continuity 
for weak solutions to $\L u = 0$ by assuming the condition $b_1, \dots, b_{m_0} \in L^q(\R^{N+1})$, with $q= Q+2$.
Our assumption on the integrability of the lower order coefficients $a_i, b_i$, with $i=1, \ldots, m_0$ and $c$ is stated as follows:
	\begin{description}
    		\item  [(H3)] $a_i, b_i,c \in L^q_\loc \left( \Omega \right)$, with $i = 1, \ldots, m_0$, for some $q > \frac34(Q+2)$. 
		Moreover, 
		\begin{equation*}
			\div a, \div b \ge 0 \qquad {\rm in } \, \, \Omega.
		\end{equation*}
	\end{description}

In general, solutions to $\L u = 0$ will be understood in the following weak sense.

\begin{definition} 
Let $\O$ be an open subset  of $\R^{N+1}$.
A weak solution to $\L u=0$ is a function $u$ such that $u,D_{m_0} u, Yu \in 
L^2_{\loc}(\O)$ and
    \begin{equation}
        \label{weak-sol2}
        \int \limits_{\O} - \langle A Du, D \phi \rangle + \phi Yu + \langle
        b, D u \rangle \phi + \langle a, D \phi \rangle u + cu \phi
        = 0, \hspace{4mm} \forall \phi \in C^{\infty}_0 ( \O ).
    \end{equation}
    In the sequel, we will also consider weak sub-solutions to $\L u=0$, 
    namely functions $u$ such that $u, D_{m_0}u, Yu \in L^2_{\loc}(\O)$ and
    \begin{equation}
        \label{subsol}
        \int_{\O} - \langle A Du, D \phi \rangle + \phi Y u + \langle b, Du 
        \rangle \phi + \langle a, D \phi \rangle u + cu \phi \ge 0, \hspace{4mm} \forall \phi \in C^{\infty}_0 (\O), 
        \phi \ge 0.
    \end{equation}
    A function $u$ is a super-solution of $\L u = 0$ 
    if $-u$ is a 
    sub-solution.
\end{definition}
We note that if $u$ is both a sub-solution and a super-solution of $\L u=0$ then it is a solution, 
i.e. $\L u = 0$ holds. Indeed, for every given $\phi \in 
C^{\infty}_0(\O)$, we may consider $\psi \in C^{\infty}_0 (\O)$ such that $\psi \ge 0$ 
and $\psi - \phi \ge 0$ in $\O$. Therefore $\L u = 0$ follows by applying 
\eqref{subsol}  to $\pm u$.

\medskip

A comparison of our result with that of Nazarov and Uralt'seva is in order. It would be natural to expect that the 
optimal lower bound for the exponent $q$ is $\frac{Q+2}{2}$. Indeed, the difficulty in considering degenerate 
equations lies in the fact that a Caccioppoli inequality gives an \emph{a priori} $L^2$ estimate for the derivatives 
$\partial_{x_1}u, \dots, \partial_{x_{m_0}}u$ of the solution $u$, that are the derivative with respect to the 
\emph{non-degeneracy} directions of $\L$. Moreover, the standard Sobolev inequality cannot be used to obtain an improvement of the 
integrability of the solution as in the non-degenerate case. For this reason we rely on a representation formula for 
the solution $u$ first used in \cite{PP}. Specifically, we represent a solution $u$ to $\L u=0$ in terms of the 
fundamental solution of $\K$. Indeed, if $u$ is a solution to $\L u=0$ in $\Omega$, then we have 
\begin{equation} \label{eq4}
	u(x,t) = \int_{\Omega} \Gamma (x,t,\xi,\tau) \K u(\xi,\tau) d\xi \, d\tau,
\end{equation}
where $\Gamma$ is the fundamental solution to $\K$ (see \eqref{eq-Gamma0} and \eqref{eq-Gamma0-b} in the sequel), 
and 
\begin{equation} \label{eq5}
	\K u = \left( \K - \L \right) u = \div \left( (A_0 - A) D u \right) - \langle b, D u \rangle + \div (a u) - cu,
\end{equation}
where we denote
 \begin{equation}
    \label{A0}
    A_0 = \begin{pmatrix} 
            \I_{m_0} & \OO\\
            \OO & \OO
           \end{pmatrix},
\end{equation}
where $\I_{m_0}$ is the identity matrix in $\R^{m_0}$, and $\OO$ 
are zero matrices.
This representation  formula provides us with a Sobolev type inequality only for weak solutions to the 
equation $\L u=0$. Specifically, we find that, for every $\Omega_1 \subset \subset \Omega_2 \subset 
\subset \Omega_3 \subset \subset \Omega$, there exist a positive constant $c_1\left( \parallel b 
\parallel_{L^{q}(\Omega)}, \Omega_1, \Omega_2 \right)$ such that 
\begin{equation*}
        \parallel u \parallel_{L^{2 \a}(\Omega_1)} \le c_1\left( \parallel a \parallel_{L^{q}(\Omega)},\parallel b \parallel_{L^{q}(\Omega)},
        \parallel c \parallel_{L^{q}(\Omega)},
        \Omega_1, \Omega_2 \right) \parallel D_{m_0} u \parallel_{L^{2}(\Omega_2)},
\end{equation*}
and, by considering $u$ as a test function, we obtain the following Caccioppoli inequality  
\begin{align*}
       \parallel D_{m_0}u \parallel_{L^2((\Omega_2)} \le  c_2\left(\parallel a \parallel_{L^{q}(\Omega)}, \parallel b \parallel_{L^{q}
       (\Omega)}, \parallel c \parallel_{L^{q}(\Omega)}, \Omega_2, \Omega_3 \right) \parallel u \parallel_{L^{2 \b}(\Omega_3)},
\end{align*}
where 
\begin{equation}
	\label{esponenti}
	\a := \frac{q (Q+2)}{q (Q-2) + 2 (Q+2)}, \qquad \qquad \b := \frac{q}{q-1} .
\end{equation}

As far as it concerns the Moser's iteration, the above inequalities are applied to a sequence of 
functions $u_k := u^{p_k}$, with $p_k \to + \infty$, in order to obtain an $L^\infty_\loc$ bound for 
the solution $u$.

We note that, the Sobolev inequality is useful to the iteration whenever $\a > 1$, and this is true if, and only 
if $q > \frac{Q+2}{2}$. Moreover, the condition $q > \frac{Q+2}{2}$ is required by Nazarov and Uralt'seva in the proof 
of the Caccioppoli inequality for non-degenerate operators. Since in our work both Sobolev and Caccioppoli inequalities 
depend on the $L^q$ norm of $a_{1}, \ldots, a_{m_{0}}, b_1, \dots, b_{m_0}, c$, we require a more restrictive condition on $q$ to improve the 
integrability of $u$. Specifically, if we combine the Sobolev and the Caccioppoli inequalities, we need to have $\a 
> \b$, and this is true if, and only if  $q > \frac{3}{4}(Q+2)$, as we require in Assumption 
$\textbf{(H3)}$.

\medskip

\medskip

We next state our main result. As we shall see in Section \ref{preliminaries}, the natural geometry underlying the 
operator $\L$ is determined by a suitable homogeneous Lie group structure on $\R^{N+1}$. Our main 
result reflect this non-Euclidean background. Let \lq \lq$\circ$" denote the Lie product on $\R^{N+1}$
defined in \eqref{grouplaw} and $\{ \d_{r} \}_{r>0}$ the family of dilations defined in \eqref{gdil}. Let us consider the cylinder:
\begin{equation*}
    \Q_1 := \left\{ (x,t) \in \R^N \times \R \hspace{1mm} : \hspace{1mm} |x| < 1, \hspace{2mm}
    |t| < 1 \right\}.
\end{equation*}
For every $z_0 \in \R^{N+1}$ and $r > 0$, we set
\begin{equation*}
    \label{cylinder}
    \Q_r (z_0) := z_0 \circ (\d_r (\Q_1)) = \left\{ z \in \R^{N+1} \hspace{1mm} : \hspace{1mm}
    z = z_0 \circ \d_r (\z), \z \in \Q_1  \right\}.
\end{equation*}

\begin{theorem}
    \label{iteration}
    Let $u$ be a non-negative weak solution to $\L u=0$ in $\O$. Let $z_0 \in \O$
    and $r, \r, \frac12 \le \r < r \le 1$, be such that $\overline{\Q_r (z_0)} \subseteq \O$. Then there exist positive constants $C= C(p, \l)$ and $\g=\g(p,q)$ such that for every $p \ne 0$, it holds
    \begin{equation}
        \label{est-it}
    	\sup_{\Q_{\r}(z_{0})} u^{p} \, \le \, \frac{C \left( 1 + \parallel a \parallel_{L^q (\Q_r(z_{0}))}^{2} + \parallel b \parallel_{L^q (\Q_r(z_{0}))}^{2} + \parallel c \parallel_{L^q (\Q_r(z_{0}))} 
	\right)^{\g}}{(r-\r)^{9 (Q+2)}} \int_{\Q_{r}(z_{0})} u^{p} ,
    \end{equation}
    where $\g = \frac{2 \a^{2} \b}{\a-1}$, with $\a$ and $\b$ defined as in \eqref{esponenti}.
\end{theorem}
\begin{remark}
	Estimate \eqref{est-it} is meaningful whenever the integral appearing in its right-hand side is finite. 
	Note that \eqref{est-it} is an estimate of the \emph{infimum} of $u$ when $p<0$. More precisely, we have that
	\begin{align}
    		\sup_{\Q_{\r}(z_{0})} u \, & \le \,  
			\frac{C^{\frac1p} \left( 1 + \parallel a \parallel_{L^q (\Q_r(z_{0}))}^{2} + \parallel b \parallel_{L^q (\Q_r(z_{0}))}^{2} + \parallel c \parallel_{L^q (\Q_r(z_{0}))} \right)^{\frac{\g}{p}} }{(r-\r)^{\frac{9 (Q+2)}{p}}}
			\left( \int_{\Q_{r}(z_{0})} u^{p} \right)^{\frac1p}, 
			\qquad \forall p > 0, \\
		\inf_{\Q_{\r}(z_{0})} u \, & \ge \,  \frac{
			C^{\frac1p} \left( 1 + \parallel a \parallel_{L^q (\Q_r(z_{0}))}^{2} + \parallel b \parallel_{L^q (\Q_r(z_{0}))}^{2} + \parallel c \parallel_{L^q (\Q_r(z_{0}))} \right)^{\frac{\g}{p}} }{(r-\r)^{\frac{9 (Q+2)}{p}}}
			\left( \int_{\Q_{r}(z_{0})} \frac{1}{u^{|p|}} \right)^{\frac1p}, 
			\hspace{4mm} \forall p < 0, \label{seconda}
	\end{align}
\end{remark}

\begin{corollary}
    \label{bounded}
    Let $u$ be a weak solution to $\L u=0$ in $\O$.  
    Then for every $p \ge 1$ we have
    \begin{equation}
        \label{boundedeq}
    	\sup_{\Q_{\r}(z_{0})} |u|^{p} \, \le \, \frac{
	C \left( 1 + \parallel a \parallel_{L^q (\Q_r(z_{0}))}^{2} + \parallel b \parallel_{L^q (\Q_r(z_{0}))}^{2} + \parallel c \parallel_{L^q (\Q_r(z_{0}))} 
	\right)^{\g} }{ (r-\r)^{9 (Q+2)}} \int_{\Q_{r}(z_{0})} |u|^{p}.
    \end{equation}
\end{corollary}

\begin{proposition}
	\label{propsubsup}
	Sub and super-solutions also verify estimate \eqref{est-it} for suitable values of $p$. More precisely, \eqref{est-it} holds for
	\begin{enumerate}
		\item $p > \tfrac12$ or $p < 0$, if $u$ is a non-negative weak sub-solution of \eqref{eq1};
		\item $p \in ]0, \tfrac12[$, if $u$ is a non-negative weak super-solution of \eqref{eq1}. 
	\end{enumerate}
\end{proposition}

\medskip

We conclude this introduction with some motivations for the study of operators $\L$ in the form \eqref{eq1}. Degenerate 
equations of the form $\L u = 0$ naturally arise in the theory of stochastic processes, kinetic theory of gases and mathematical 
finance. For instance, if $\left(W_t \right)_{t \ge 0}$ denotes a real Brownian motion, then the simplest non-trivial 
Kolmogorov operator 
\begin{equation*}
    \frac{1}{2} \p_{vv} + v \p_x + \p_t, \qquad t \ge 0, \hspace{1mm}
    (v,x) \in \R^2
\end{equation*}
is the infinitesimal generator of the classical Langevin's stochastic equation
that describes the position $X$ and the velocity $V$ of a particle in the phase
space (cf. \cite{L})
\begin{equation*}
    \begin{cases}
    d V_t &= \hspace{2mm} d W_t, \\
    d X_t &= \hspace{2mm} V_t \hspace{.5mm} dt.
    \end{cases}
\end{equation*}
Notice that in this case we have $1=m_0 < N=2$. 

Linear Fokker-Planck equations (cf. \cite{DV} and \cite{R}), non-linear 
Boltzmann-Landau equations (cf. \cite{L} and \cite{CE}) and non-linear 
equations for Lagrangian stochastic models commonly used in the simulation of
turbulent flows (cf. \cite{BJT}) can be written in the form
\begin{equation}
    \label{turb}
    \sum \limits_{i,j=1}^n \p_{v_i} (a_{ij} \p_{v_j} f) + 
    \sum \limits_{j=1}^n v_j\p_{x_j} f + \p_t f = 0,
    \qquad t \ge 0, v \in \R^n, x \in \R^n
\end{equation}
with the coefficients $a_{ij}= a_{ij}(t,v,x,f)$ that may depend on the solution
$f$ through some integral expressions. It is clear that equation \eqref{turb} is a particular case of 
$\L u=0$ with $n=m_0 < d =2n$ and
\begin{equation*}
    B = \begin{pmatrix}
            \OO_n & \OO_n \\
            \I_n  & \OO_n
        \end{pmatrix}
\end{equation*}
where $\I_n$ and $\OO_n$ denote the $(n \times n)-$identity matrix and the $(n \times n)-$zero matrix, respectively.

In mathematical finance, equations of the form $\L u = 0$ appear in various models for pricing of path-dependent derivatives such as Asian options (cf., for instance, \cite{BPVE} \cite{PA}), stochastic velocity models (cf. \cite{HR} \cite{P}) and in theory of stochastic utility (cf. \cite{ABM} \cite{AP}).

\medskip

This note is organized as follows. In Section \ref{preliminaries} we recall some 
known facts about operators $\L$ and $\K$, and we give some preliminary 
results. In Section \ref{secsobolev} we prove Theorem \ref{sobolev} and Proposition \ref{pcaccB1},
which is an intermediate result (Caccioppoli type inequality for weak solutions to $\L u = 0$) needed for the bootstrap 
argument. Finally, in Section \ref{moser} we deal with the Moser's iterative method.

\section{Preliminaries}
\label{preliminaries}

In this Section we recall notation and results we need in order to deal with the non-Euclidean 
geometry underlying the operators $\L$ and $\K$. We refer to the articles \cite{CPP} and 
\cite{LP} for a comprehensive treatment of this subject. The operator $\K$ is invariant 
with respect to a Lie product on $\R^{N+1}$. More precisely, we let 
\begin{equation}
    \label{exp}
    E(s) = \exp (-s B), \qquad s \in \R,
\end{equation}
and we denote by $\ell_{\zeta}, \zeta \in \R^{N+1}$, the left translation 
$\ell_{\zeta} (z) = \zeta \circ z$ in the group law
\begin{equation}
    \label{grouplaw}
    (x,t) \circ (\xi, \tau) = (\xi + E(\tau) x, t + \tau ), \hspace{5mm} (x,t).
    (\xi, \tau) \in \R^{N+1}.
\end{equation}
Thus we have
\begin{equation*}
    \K \circ \ell_{\zeta} = \ell_{\zeta} \circ \K.
\end{equation*}
This means that, if $v(x,t) = u \big((\xi, \tau) \circ (x,t)\big)$ and $g(x,t) = f \big((\xi, \tau) \circ (x,t)\big)$, 
we have 
\begin{equation*}
    \K u = f \quad \iff \quad \K v = g.
\end{equation*}

We recall that, by \cite{LP} (Propositions 2.1 and 2.2), assumption \textbf{(H2)}
is equivalent to assume that, for some basis on $\R^N$, the matrix $B$ has the 
canonical form
\begin{equation}
    \label{B}
    B =
    \begin{pmatrix}
        *   &   *   & \ldots &    *    &   *   \\  
        B_1   &    *   & \ldots &    *    &   *  \\
        \OO    &    B_2  & \ldots &  *    &   *    \\
        \vdots & \vdots & \ddots & \vdots & \vdots \\
        \OO    &  \OO    &    \ldots & B_\k    & *
    \end{pmatrix}
\end{equation}
where every $B_k$ is a $m_{k} \times m_{k-1}$ matrix of rank $m_j$, $j = 1, 2, \ldots, \k$ 
with 
\begin{equation*}
    m_0 \ge m_1 \ge \ldots \ge m_\k \ge 1 \hspace{5mm} \text{and} \hspace{5mm} 
    \sum \limits_{j=0}^\k m_j = N
\end{equation*}
and the blocks denoted by \lq \lq *" are arbitrary. In the sequel we shall assume that $B$ has the 
canonical form \eqref{B}.

We denote by $\Gamma ( \cdot, \zeta)$ the fundamental solution of $\K$ in \eqref{eq2} 
with pole in $\zeta \in \R^{N+1}$. An explicit expression of $\Gamma ( \cdot, \zeta )$
has first been constructed by Kolmogorov \cite{K1} for operators in the form \eqref{turb}, 
then by H\"ormander in \cite{H} under more general conditions
\begin{equation} \label{eq-Gamma0}
    \Gamma (z,\zeta) = \Gamma (\zeta^{-1} \circ z, 0 ), \hspace{4mm} 
    \forall z, \zeta \in \R^{N+1}, \hspace{1mm} z \ne \zeta,
\end{equation}
where
\begin{equation} \label{eq-Gamma0-b}
    \Gamma ( (x,t), (0,0)) = \begin{cases}
    \frac{(4 \pi)^{-\frac{N}{2}}}{\sqrt{\text{det} C(t)}} \exp \left( - 
    \frac{1}{4} \langle C^{-1} (t) x, x \rangle - t \, \trace (B) \right), \hspace{3mm} & 
    \text{if} \hspace{1mm} t > 0, \\
    0, & \text{if} \hspace{1mm} t < 0,
    \end{cases}
\end{equation}
 and
 \begin{equation*}
     C(t) = \int \limits_0^t \hspace{1mm} E(s) \, A_0 \, E^T(s) \, ds,
 \end{equation*}
where $E(\cdot)$ is the matrix defined in \eqref{exp}. Note that assumption 
 \textbf{(H2)} implies that $C(t)$ is strictly positive for every $t>0$ 
 (see \cite{LP}, Proposition A.1). 
 
Among the operators $\K$ where the matrix $B$ is of the form \eqref{B}, the ones for 
which the $*-$blocks are equal to zero play a central role. Indeed, let us consider the 
principal part operator $\K =  \D_{m_0} + Y_0$, where $Y_0 = \langle  B_0 x, D \rangle - 
\p_t$ and 
\begin{equation}
    \label{B_0}
    B_0 =
    \begin{pmatrix}
        \OO  &   \OO  & \ldots &    \OO   &   \OO  \\  
        B_1   &    \OO  & \ldots &    \OO   &   \OO  \\
        \OO    &    B_2  & \ldots &  \OO    &   \OO   \\
        \vdots & \vdots & \ddots & \vdots & \vdots \\
        \OO    &  \OO    &    \ldots & B_\k    &   \OO
    \end{pmatrix}\end{equation}
The operator $\K_{0}$ is invariant with respect to the dilations defined as 
 \begin{equation}
    \label{gdil}
    \d_{r} = \text{diag} ( r \I_{m_0}, r^3 \I_{m_1}, \ldots, r^{2\k+1} \I_{m_\k}, 
    r^2), \qquad \qquad r > 0.
\end{equation}
In order to explain the importance of this invariance property we introduce for every 
positive $r$ the scaled operator
\begin{equation*}
    \K_{r} = r^2 \left( \d_r \circ \K \circ \d_{\frac{1}{r}} \right) .
\end{equation*}
In order to explicitly write $\K_r$ we note that, if 
\begin{equation}
    \label{B1}
    B =
      \begin{pmatrix}
         B_{0,0} &    B_{0,1}  & \ldots &     B_{0, \k - 1}  &   B_{0, \k }  \\  
        B_1   &    B_{1,1}  & \ldots &     B_{\k - 1 , 1}  &   B_{\k , 1}  \\
        \OO    &    B_2  & \ldots &   B_{\k - 1, 2}   &  B_{\k , 2}   \\
        \vdots & \vdots & \ddots & \vdots & \vdots \\
        \OO    &  \OO    &    \ldots & B_\k    &   B_{\k,\k}
    \end{pmatrix},
\end{equation}
where $B_{i,j}$ are the $m_i \times m_j$ blocks denoted by $\lq \lq * "$ in \eqref{B},
then we can rewrite $\K_r$ as follows
\begin{equation}
    \label{Kr}
    \K_r = \div (A_0 D) + Y_r,
\end{equation}
where
\begin{equation}
    \label{Y_r}
    Y_r := \langle B_r \, x, D \rangle - \p_t 
\end{equation}
and $B_r := r^2 \, D_r \, B \, D_{\frac{1}{r}},$  i.e.
\begin{equation*}
    B_r = 
    \begin{pmatrix}
        r^2 B_{0,0} &   r^4 B_{0,1}  & \ldots &    r^{2\k} B_{0, \k - 1}  &  r^{2\k+2} B_{0, \k }  \\  
        B_1   &   r^2 B_{1,1}  & \ldots &    r^{2\k-2} B_{\k - 1 , 1}  &   r^{2\k} B_{\k , 1}  \\
        \OO    &    B_2  & \ldots &  r^{2\k-4} B_{\k - 1, 2}   &  r^{2\k-2} B_{\k , 2}   \\
        \vdots & \vdots & \ddots & \vdots & \vdots \\
        \OO    &  \OO    &    \ldots & B_\k    &   r^2 B_{\k,\k}
    \end{pmatrix}.
\end{equation*}

Note that 
\begin{equation*}
    B_r=B \qquad {\rm for \, \, every \,} r>0 
\end{equation*}
if, and only if $B_{j,k}=\OO$ with $j \le k$. In this case, if $v(x,t) = u \big(\d_{r} (x,t)\big)$ and $g(x,t) = f \big(\d_{r} (x,t)\big)$, then
\begin{equation*}
    \K u = f \quad \iff \quad \K v = r^2 g.
\end{equation*}
Since $K_0$ is the blow-up limit of $K_r$, the dilation group $(\d_r)_{r>0}$ plays a central role also for non-dilation invariant operators.

We next introduce a norm which is homogeneous of degree $1$ with respect to the dilations 
$(\d_{r})_{r>0}$ and a corresponding quasi-distance which is invariant with respect to the 
translation group for the case of $*-$blocks equal to zero.
\begin{definition}
    \label{hom-norm}
    Let $\a_1, \ldots, \a_N$ be the positive integers such that
    \begin{equation*}
       \diag \left( r^{\a_1}, \ldots, r^{\a_N}, r^2 \right) = \d_r.
    \end{equation*}
    If $\parallel z \parallel = 0$ we set $z=0$ while, if $z \in \R^{N+1} 
    \setminus \{ 0 \}$ we define $\parallel z \parallel = r$ where $r$ is the 
    unique positive solution to the equation
    \begin{equation*}
        \frac{x_1^2}{r^{2 \a_1}} + \frac{x_2^2}{r^{2 \a_2}} + \ldots 
        + \frac{x_N^2}{r^{2 \a_N}} + \frac{t^2}{r^4} = 1.
    \end{equation*}
    We define the quasi-distance $d$ by
    \begin{equation*}
        d(z, w) = \parallel z^{-1} \circ w \parallel , \hspace{5mm} z, w \in 
        \R^{N+1}.
    \end{equation*}
\end{definition}
\begin{remark}
    The Lebesgue measure is invariant with respect to the translation group 
    associated to $\K$, since $\det E(t) = e^{t \hspace{1mm} \text{\rm trace} \,
    B} = 1$, where $E(t)$ is the exponential matrix of equation
    \eqref{exp}. Moreover, since $\det \d_r = r^{Q+2}$, we also have 
    \begin{equation*}
        \meas \left( \Q_r(z_0) \right) = r^{Q+2} \meas \left( \Q_1(z_0) \right), \qquad \forall
        \ r > 0, z_0 \in \R^{N+1},
    \end{equation*}
where
\begin{equation}
       \label{hom-dim}
         Q = m_0 + 3 m_1  + \ldots + (2\k+1) m_\k.
\end{equation}
The natural number $Q+2$ is usually called the \textit{homogeneous dimension of} $\R^{N+1}$ \textit{with respect to} 
$(\d_{r})_{r > 0}$.
\end{remark}
\begin{remark}
    The norm $\parallel \cdot \parallel$ is homogeneous of degree $1$ with respect to $(\d_r )_{r>0}$, that is
    \begin{equation*}
        \parallel \d_\r (x,t) \parallel = \r  \parallel  (x,t) \parallel
        \qquad \forall \r >0 \hspace{2mm} \text{and} \hspace{2mm} (x,t) \in \R^{N+1}.
    \end{equation*}
    Actually in $\R^{N+1}$ all the norms, that are $1$-homogeneous with respect to 
    $(\d_r)_{r>0}$, are equivalent. In particular, the norm introduced in Definition 
    \ref{hom-norm} is equivalent to the following one
    \begin{equation*}
        \parallel (x,t) \parallel_1 = |x_1|^{\frac{1}{\a_1}} + \ldots + 
        |x_N|^{\frac{1}{\a_N}} + | t |^{\frac{1}{2}},
    \end{equation*}
    where the homogeneity with respect to $(\d_r)_{r>0}$ can easily be showed. We 
    prefer the norm of Definition \ref{hom-norm} to $\parallel \cdot \parallel_1$ 
    because its level sets (spheres) are smooth surfaces.
\end{remark}

When $\K_0$ is dilation invariant with respect to $(\d_r )_{r>0}$, also its fundamental solution $\G_0$ is a homogeneous 
function of degree $- Q$, namely
\begin{equation*}
         \Gamma_0 \left( \d_{r}(z), 0 \right) = r^{-Q} \hspace{1mm} \Gamma_0
        \left( z, 0 \right), \hspace{5mm} \forall z \in \R^{N+1} \setminus
        \{ 0 \}, \hspace{1mm} r > 0.
\end{equation*}
This property implies an $L^p$ estimate for Newtonian potential (c. f. for instance 
\cite{FO}).
  \begin{proposition}
    Let $\a \in ]0, Q+2[$ and let $G \in C (\R^{N+1} \setminus \{ 0 \})$ be a 
    $\d_\l-$homogeneous function of degree $\a - Q - 2$. If $f \in L^p 
    (\R^{N+1})$ for some $p \in ]1, +\infty[$, then the function
    \begin{equation*}
        G_f (z) := \int_{\R^{N+1}} G ( \z^{-1} \circ z) f(\z) d\z,
    \end{equation*}
    is defined almost everywhere and there exists a constant $c = c(Q, p)$ such
    that 
    \begin{equation*}
        \parallel G_f \parallel_{L^q(\R^{N+1}} \le c \max_{\parallel z 
        \parallel = 1} |G(z)| \parallel f \parallel_{L^p (\R^{N+1})},
    \end{equation*}
    where $q$ is defined by
    \begin{equation*}
        \frac{1}{q} = \frac{1}{p} - \frac{\a}{Q+2}.
    \end{equation*}
 \end{proposition}

It is known that homogeneous operators provide a good approximation of the 
non-homogeneous ones. In order to be more specific, let us consider a homogeneous operator of 
the form  
\begin{equation*}
    \K_0 = \div (A_0 D) + \langle B_0 x,  D \rangle - \p_t,
\end{equation*} 
where $B_0$ is the matrix in \eqref{B_0}, and denote by $\G_{0}$ the fundamental 
solution of $\K_{0}$. If $\G$ denotes the fundamental solution of $\K$ defined in \eqref{eq-Gamma0-b}, then, for every $M>0$, there exists a positive constant $c$ such that 
\begin{equation}
    \label{estg}
    \frac{1}{c} \, \G_{0} \, \le \, \G(z) \, \le \, c \, \G_{0} (z)
\end{equation}
for every $z \in \R^{N+1}$ such that $\G_{0} (z) \ge M$ (see \cite{LP}, Theorem 3.1).

We define the $\G-$\textit{potential} of the function $f \in L^1(\R^{N+1})$ 
as follows
 \begin{equation}
     \label{L0p}
     \G (f) (z) = \int_{\R^{N+1}} \G (z, \z ) f(\z) d\z, \qquad 
     z \in \R^{N+1}.
 \end{equation}

 We also remark that the potential $\G (D_{m_0} f): \R^{N+1} \longrightarrow \R^{m_0}$ is 
 well-defined for any $f \in L^p (\R^{N+1})$, at least in the distributional sense, that is 
 \begin{equation}
     \label{pot}
     \G (D_{m_0} f ) (z) := - \int_{\R^{N+1}} D^{(\x)}_{m_0} \G (z, \x) \,
     f(\x) \, d\x,
 \end{equation}
 where $D^{(\x)}_{m_0} \G (x,t, \x, \t)$ is the gradient with respect to $\x_1, 
 \ldots, \x_{m_0}$. Based on \eqref{estg}, in \cite{CPP} are proved potential estimates for non-dilation invariant operators.
 \begin{theorem}
    \label{corollary}
    Let $f \in L^p (\Q_r)$. There exists a positive constant $c = c(T,B)$ 
    such that
    \begin{align}
        \parallel \G (f) \parallel_{L^{p**}(\Q_r)} &\le c \parallel f 
        \parallel_{L^{p}(\Q_r)}, \label{stima1} \\ 
        \parallel \G (D_{m_0} f) \parallel_{L^{p*}(\Q_r)} &\le c 
        \parallel f \parallel_{L^{p}(\Q_r)}, \label{stima2}
    \end{align}
where $\frac{1}{p*}= \frac{1}{p} - \frac{1}{Q+2}$ and $\frac{1}{p**}= 
\frac{1}{p} - \frac{2}{Q+2}$.
\end{theorem}

We can use the fundamental solution $\G$ as a test function in the definition of sub and 
super-solution. The following result extends Lemma 2.5 in \cite{PP} and Lemma 3 in \cite{CPP}.
\begin{lemma}
    \label{lemma2.5}
    Let $v$ be a non-negative weak sub-solution to $\L u = 0$  
    in $\O$. For 
    every $\phi \in C^{\infty}_0 (\O)$, $\phi \ge 0$, and for almost every $z 
    \in \R^{N+1}$, we have 
    \begin{align*}
        \int_{\O} - \langle A Dv, D (\G (z, \cdot) \phi ) \rangle &+
        \G (z, \cdot) \phi Yv + \\
        &- \langle a, D( \G(z, \cdot) \phi) \rangle v - \langle b, D (\G (z, \cdot) \phi) \rangle v + c u \G(z, \cdot) \phi 
        \ge 0.
    \end{align*}
    An analogous result holds for weak super-solutions to $\L u = 0$.
\end{lemma}
\begin{proof}
    We define the cut-off function $\cc_{\r,r} \in C^\infty (\R^+)$
    \begin{equation}
        \label{chi}
        \cc_{\r,r} (s) = \begin{cases}
            0 \hspace{6mm} & \text{if} \hspace{2mm} s \ge r, \\
            1 \hspace{6mm} & \text{if} \hspace{2mm} 0 \le s < \r,
        \end{cases}
        \hspace{15mm}
        |\cc_{r,\r}'| \le \frac{2}{r-\r}
    \end{equation}
    with $\frac{1}{2} \le \r < r \le 1 $. Moreover, for every $\e < 0$ we 
    define
    \begin{equation}
        \psi_\e (x,t) = 1 - \cc_{\e,2\e} (\parallel 
        (x,t) \parallel).
    \end{equation}
    Because $v$ is a weak sub-solution, then by \eqref{subsol} for every $\e >
    0$ and $z \in \R^{N+1}$ we have
    \begin{align*}
        0 &\le \int_{\O} - \left[ \langle \hspace{1mm} A Dv, D ( \G (z,\cdot) \phi(\z) \psi_\e(z, \cdot)) 
        		\, \rangle \, + \, \G (z, \cdot) \phi(\z) \psi_\e(z, \cdot) \, Y v \right] d \z \\ 
        & \qquad + \, \int_\O \left[ \langle b, Dv \rangle \, \G (z, \cdot) \phi(\z) \psi_\e(z, \cdot) \, + \langle a, D( \G(z, \cdot) \phi(\z) 
        		\psi_\e(z, \cdot)) \rangle v + c u  \,\G(z, \cdot) \phi(\z) \psi_\e(z, \cdot) \right] d \z \\
        &= \, - \, I_{1,\e} (z) \, + \, I_{2, \e}  (z) \, - \, I_{3, \e} (z) \, + \, I_{4, \e}  (z) + I_{5, \e} (z)
    \end{align*}
    where
    \begin{align*}
        I_{1,\e} (z) &= \int_\O \langle \hspace{1mm} A Dv, D  \G (z, \cdot) 
        \hspace{1mm} \rangle \phi(\z) \psi_\e(z, \z) \hspace{1mm} d\z \\
        I_{2,\e} (z) &= \int_\O \G (z, \cdot) \phi(\z) \hspace{1mm} \left( 
        \hspace{1mm} - \langle 
        \hspace{1mm} A Dv, D\phi(\z) \hspace{1mm} \rangle + \phi(\z) Yv 
        \hspace{1mm} \right) \hspace{1mm} d\z  \\
        I_{3,\e} (z) &= \int_\O \langle \hspace{1mm} A Dv, D \psi_\e(z, \cdot) 
        \hspace{1mm} \rangle \phi(\z) \G (z,\cdot) \hspace{1mm} d\z \\
        I_{4,\e} (z) &= \int_\O \langle  b, Dv  \rangle \, \G (z, \cdot) \phi(\z) \psi_\e(z, \cdot) \, d\z  
        				+ \int_\O \langle a, D( \G(z, \cdot) \phi) \rangle v \, d \z \\
	I_{5,\e} (z) &= \int_\O c u \, \G(z, \cdot) \phi(\z) \psi_\e(z, \cdot) \, \, d \z
    \end{align*}
    Keeping in mind Theorem \ref{corollary}, it is clear that the integral 
    which defines $I_{i,\e}(z), \hspace{2mm} i = 1, 2,3$ is a potential and it 
    is convergent for almost every $z \in \R^{N+1}$. Thus, by a similar 
    argument to the one used in \cite{PP} to prove Lemma 2.5 (pg. $403-404$), 
    we get that for almost every $z \in \R^{N+1}$
    \begin{align*}
        \lim \limits_{\e \rightarrow 0^+} I_{1,\e} (z) &= \int_{\O} \langle A 
        Dv, D(\G (z, \cdot)) \rangle \phi(\z) \hspace{1mm} d\z    \\
        \lim \limits_{\e \rightarrow 0^+} I_{2,\e} (z) &= \int_{\O} \G (z, 
        \cdot) \left( - \hspace{1mm} \langle A Dv, D \phi(\z) \rangle + 
        \phi(\z) Yv \right) \hspace{1mm} d\z \\
         \lim \limits_{\e \rightarrow 0^+} I_{3,\e} (z) &= 0 .
    \end{align*}
   Let us consider the term $I_{4, \e}$. We integrate by 
   parts and we consider assumption \textbf{(H3)}:
    \begin{align*}
        I_{4,\e} = &- \int_\O \div b \, \G (z, \cdot) \phi(\z) \cc_\e(z, \cdot)  v \, d \z \,
        			- \int_\O \langle \, b, D  \left( \G (z,\cdot) \phi(\z) \cc_\e(z, \cdot)   \right) \, \rangle \, v \, d \z \, \\
			&- \int_\O \div a \, \G (z, \cdot) \phi(\z) \cc_\e(z, \cdot)  v \, d \z \,
        			- \int_\O \langle \, a, D  \left( \G (z,\cdot) \phi(\z) \cc_\e(z, \cdot)   \right) \, \rangle \, v \, d \z \, \\
        \le &- \int_\O \langle \, b, D \left( \G (z,\cdot) \phi(\z)  \cc_\e(z, \cdot)  \right) \, \rangle \, v \, d \z \,
        			- \int_\O \langle \, a, D  \left( \G (z,\cdot) \phi(\z) \cc_\e(z, \cdot)   \right) \, \rangle \, v \, d \z
    \end{align*}
    We are left with the estimate of a potential and in order to do so we would
    like to use Theorem \ref{corollary}. Because $a_i, b_i \in L^q_\loc (\O)$, with $i = 1, \ldots, m_0$ and $v
    \in L^2_\loc(\O)$, we have that 
    \begin{equation*}
        |a| \, | \G (z, \cdot) | \, | \phi | \, |D_{m_0} v| \, , \,  \,|b| \, |  \G (z, \cdot) | \, | \phi | \, 
        |D_{m_0} v|  \in L^{2\a}_\loc(\O) 
    \end{equation*}
    where $\a$ is defined as in \eqref{esponenti}. This yields, for every $\e > 0$
    \begin{align*}
    	|\langle \hspace{1mm} a, D \left( \G (z,\cdot) \phi(\z) \cc_\e(z, 
        \cdot) \right) \hspace{1mm} \rangle \hspace{1mm} v | \le |\langle 
        \hspace{1mm} a, D \left( \G (z,\cdot) \phi(\z)  \right) \hspace{1mm} 
        \rangle \hspace{1mm} v | \in L^1_\loc(\O),\\
        |\langle \hspace{1mm} b, D \left( \G (z,\cdot) \phi(\z) \cc_\e(z, 
        \cdot) \right) \hspace{1mm} \rangle \hspace{1mm} v | \le |\langle 
        \hspace{1mm} b, D \left( \G (z,\cdot) \phi(\z)  \right) \hspace{1mm} 
        \rangle \hspace{1mm} v | \in L^1_\loc(\O).
    \end{align*}
    Thus, by the Lebesgue convergence theorem, we get for a.e. $z \in \R^{N+1}$
    \begin{align*}
         \lim \limits_{\e \rightarrow 0^+} \left[ \int_\O - \langle \hspace{1mm} b, D 
         \left( \G (z,\cdot) \phi(\z) \cc_\e(z, \z)   \right) \hspace{1mm} 
         \rangle \hspace{1mm} v \, d\z \, - \, \int_\O \langle \, a, D  \left( \G (z,\cdot) \phi(\z) \cc_\e(z, \cdot)   \right) \, \rangle \, v
         \right] \, d \z  = \\
         = - \int_\O \langle \hspace{1mm} b, D \left( 
         \G (z,\cdot) \phi(\z) \right) \hspace{1mm} \rangle \hspace{1mm} v  -  \int_\O \langle \, a, D  \left( \G (z,\cdot) \phi(\z) 
           \right) \, \rangle \, v  \, d \z.
    \end{align*}
    Now, we are left with an estimate of the term $I_{5,\e}$, which is a $\G-$potential such that
    \begin{equation*}
        |c| \, |\G (z, \cdot) | \, | \phi | \, |v|  \in L^{2\a}_\loc(\O) .
    \end{equation*}    
    Thus, we have that 
    \begin{align*}
	| c u \, \G(z, \cdot) \phi(\z) \psi_\e(z, \cdot) | \le |	c u \, \G(z, \cdot) \phi(\z) | \in L^1_\loc (\O).
    \end{align*}
    Then we can apply the Lebesgue convergence theorem and we get for a. e. $z \in \R^{N+1}$
    \begin{equation*}
    	 \lim \limits_{\e \rightarrow 0^+} \int_\O c v \G (z,\cdot) \phi(\z) \cc_\e(z, \z) \, d\z \ = \int_\O c v \G (z,\cdot) \phi(\z)  \, d\z .
    \end{equation*}
\end{proof}

\section{Sobolev and Caccioppoli Inequalities}
\label{secsobolev}
In this Section we give proof of a Sobolev inequality and a Caccioppoli inequality for weak solutions to $\L u = 0$. 
We start considering the Sobolev inequality and we remark that it holds true for every $q > \frac{Q+2}{2}$.
\begin{theorem}[Sobolev Type Inequality for sub-solutions]  \label{sobolev}
    Let \textbf{(H1)}-\textbf{(H2)} hold. Let \\ $a_1, \ldots, a_{m_0}, b_1, \ldots, b_{m_0}, c \in L^q_\loc (\O)$,
    for some $q > (Q+2)/2$, and $\div a, \div b \ge 0$ in $\O$. Let $v$ be a non-negative weak 
    sub-solution of $\L u = 0$ in $\Q_1$. Then there exists a constant $C = C(Q,\l) > 0$ such that $v \in L^{2\a}_{\loc} (\Q_1)$, and the following statement holds
    \begin{align*}
        \parallel v \parallel_{L^{2 \a}(\Q_\r(z_0))} \le &C \cdot \left(  \parallel a  \parallel_{L^{q}(\Q_r(z_0))} + 
        \parallel b \parallel_{L^{q}(\Q_r(z_0))} + 1 + \frac{1}{r - \r} \right) \parallel D v  \parallel_{L^{2}(\Q_r(z_0))} + \\
        &+ C \cdot \left( \parallel c \parallel_{L^{q}(\Q_r(z_0))} + \frac{\r + 1}{\r(r - \r)} \right)
        \parallel  v \parallel_{L^{2}(\Q_r(z_0))} 
    \end{align*}
    for every $\r, r$ with $\frac{1}{2} \le \r < r \le 1$ and for every $z_0 \in \O$, where $\a =\a(q)$ is defined in \eqref{esponenti}.
\end{theorem}

\begin{proof}
    Let $v$ be a non-negative weak sub-solution to $\L u = 0$. We represent $v$ in terms of the fundamental 
    solution $\G$. To this end, we consider the cut-off function $\cc_{\r,r}$ defined in \eqref{chi}
    for $\frac{1}{2} \le \r < r \le 1$. Then we consider the following test function
    \begin{equation}
        \label{cutoff}
        \psi (x,t) = \cc_{\r,r} ( \parallel (x,t) \parallel )
    \end{equation}
    and the following estimates hold true
    \begin{equation}
        \label{ineq}
        |Y \psi| \le \frac{c_0}{\r (r-\r)},
        \hspace{8mm}
        |\p_{x_j} \psi| \le \frac{c_1}{r-\r}
        \hspace{2mm} \text{for } j = 1, \ldots, m_0
    \end{equation}
    where $c_0$, $c_1$ are dimensional constants.
    For every $z \in \Q_\r$, we have
    \begin{align}
        \label{representation}
        v(z) &= v \psi(z) \\ \nonumber
          &= \int_{\Q_r} \left[ \langle A_0 D(v \psi), D\G (z, \cdot) \rangle - \G (z, \cdot) Y(v \psi) \right] (\z) d(\z) \\ \nonumber
          &=  I_0(z) + I_1(z) + I_2(z) + I_3(z)
    \end{align}
    where
    \begin{align*}
        I_0 (z) &= - \int_{\Q_r} \left[ \langle a, D ( \psi \G(z, \cdot)) \rangle  v 
                \right](\z) d \z \hspace{1mm}
                 - \int_{\Q_r} \left[ \langle b, D (\psi  \G(z, \cdot) ) \rangle v 
                \right](\z) d \z \hspace{1mm} 
                + \int_{\Q_r} \left[ c v \G (z, \cdot) \psi \right](\z) d \z \\
        I_1 (z) &= \int_{\Q_r} \left[ \langle A_0 D\psi, D \G (z, \cdot ) \rangle 
             v \right](\z) d \z \hspace{1mm} - \hspace{1mm} \int_{\Q_r} \left[ 
            \G (z, \cdot) v Y\psi \right](\z) d\z = I_1^{'} + I_1^{''} , \\
        I_2 (z) &= \int_{\Q_r} \left[ \langle (A_0 - A) Dv, D \G (z, \cdot ) 
            \rangle \psi \right](\z) d \z \hspace{1mm} - \hspace{1mm}
            \int_{\Q_r} \left[ \G (z, \cdot ) \langle A Dv, D \psi \rangle 
            \right](\z) d \z \\
        I_3 (z) &= \int_{\Q_r} \left[ \langle A Dv, D( \G (z, \cdot ) \psi)
            \rangle \right](\z) d \z \hspace{1mm} - \hspace{1mm}
            \int_{\Q_r} \left[ \left( \G (z, \cdot ) \psi \right) Yv
            \right](\z) d \z \hspace{1mm} + \\
            &\hspace{4mm} + \, \int_{\Q_r} \left[ \langle a, D (\G (z, \cdot) \psi ) 
            \rangle v \right](\z) d \z  \, + \, \int_{\Q_r} \left[ \langle b, D (\G (z, \cdot) \psi ) 
            \rangle v \right](\z) d \z -\int_{\Q_r} \left[ c v \G (z, \cdot) \psi \right](\z) d \z 
    \end{align*}
    Since $v$ is a non-negative weak sub-solution to $\L u = 0$, it follows from Lemma 
    \ref{lemma2.5} that $I_3 \le 0$, then
    \begin{equation*}
        0 \le v(z) \le I_0(z) + I_1(z) + I_2(z)  \hspace{4mm} \text{for a.e. } z \in \Q_\r.
    \end{equation*}
    To prove our claim is sufficient to estimate $v$ by a sum of $\G-$potentials.
    
    We start by estimating $I_0$. In order to do so, we recall that
    \begin{equation*}
        \langle a, Dv \rangle, \, \langle b, Dv \rangle, \, c v \,  \in L^{2 \frac{q}{q+2}} \hspace{4mm} 
        \text{for } b \in L^q, \hspace{2mm} q > \frac{Q+2}{2} \hspace{2mm}
        \text{and } Dv \in L^2.
    \end{equation*}
    Thus by Theorem \ref{corollary} we get
    \begin{equation*}
        \G * \langle a, Dv \rangle , \G * \langle b, Dv \rangle , \G * (c v ) \in L^{2\a},
    \end{equation*}
    where $\a = \a(q)$ is defined in \eqref{esponenti}.
    When $q \le (Q+2)$ we have that $\a \le 2^{**}$. Moreover, thanks to estimate
    \eqref{stima1}, we have
    \begin{align*}
        \parallel I_0 (\z) \parallel_{L^{2\a}(\Q_\r)} &\le 
        \text{meas}(\Q_\r)^{2/Q} \parallel I_0 (\z) 
        \parallel_{L^{2^{**}}(\Q_\r)} \\
        &= \text{meas}(\Q_\r)^{2/Q}
        \parallel \G * \left( \langle a,  D_{m_0}v \rangle \psi \right) + \G * \left( \langle b,  D_{m_0}v \rangle \psi \right)  + \G * \left( c v \psi \right) 
        \parallel_{L^{2^{**}}(\Q_\r)} \\
        &\le C \cdot \left( \parallel a \parallel_{L^q (\Q_\r)} + \parallel b \parallel_{L^q (\Q_\r)} \right) \parallel D_{m_0}v \parallel_{L^2 (\Q_\r)}
        + C \cdot \parallel c \parallel_{L^q (\Q_\r)} \parallel v \parallel_{L^2 (\Q_\r)}.
    \end{align*}
    We prove an estimate for the term $I_1$. $I_1'$ can be 
    estimated by \eqref{stima2} of Theorem \ref{corollary} as follows
    \begin{equation*}
        \parallel I_1 ' \parallel_{L^{2\a}(\Q_\r)} \le C
        \parallel I_1 ' \parallel_{L^{2^*}(\Q_\r)} \le C \parallel v 
        D_{m_0} \psi \parallel_{L^2 (\R^{N+1})} \le \frac{C}{r - \r} 
        \parallel v \parallel_{L^2(\Q_\r)},
    \end{equation*}
    where the last inequality follows from \eqref{ineq}. To estimate $I_1 ''$
    we use \eqref{stima1}
    \begin{align*}
        \parallel I_1 '' \parallel_{L^{2\a}(\Q_\r)} &\le C
        \parallel I_1 '' \parallel_{L^{2^*}(\Q_\r)} \le 
        \text{meas}(\Q_\r)^{2/Q}
        \parallel I_1 '' \parallel_{L^{2^{**}}(\Q_\r)} \\ 
        &\le 
        C \parallel v Y \psi \parallel_{L^2(\R^{N+1})} \le \frac{C}{\r (r-\r)} 
        \parallel v \parallel_{L^2 (\Q_\r)}.
    \end{align*}
    We can use the same technique to prove that 
    \begin{equation*}
        \parallel I_2 \parallel_{L^{2\a}(\Q_\r)} \le C \left( 1 + \frac{1}{r - \r} \right) \parallel
        D v \parallel_{L^2(\Q_\r)},
    \end{equation*}
    for some constant $C=C(Q, \l)$.  
    
    A similar argument proves the thesis when $v$ is a super-solution to $\L u = 0$. In this 
    case we introduce the following auxiliary operator
    \begin{equation}
    	\K = \div (A_0 \, D ) + \widetilde{Y}, \qquad \qquad 
	\widetilde{Y} \equiv - \langle x, B D \rangle - \p_t \, .
    \end{equation}
    Then we proceed analogously as in \cite{PP}, Section 3, proof of Theorem 3.3.
\end{proof}

\medskip

Finally, we give proof of a Caccioppoli inequality for weak solutions to $\L u = 0$.
\begin{proposition}
    \label{pcaccB1}
    Let \textbf{(H1)}-\textbf{(H3)} hold. Let $u$ be a non-negative 
    weak solution of $\L u = 0$ in $\Q_1$. Let $p \in \R$, $p \ne 0$, $p \ne 1/2$ 
    and let $r, \r$ be such that $\frac{1}{2} \le \r < r \le 1$. Then there exists
    a constant $C$ such that 
    \begin{align*}
    &\frac1\l \, \parallel D v \parallel^2_{L^2(\Q_\r)} \, \le \\
    &\le \left[ \frac{C \, p}{2 \l} \frac{1}{(r-\r)^2} \, + \, \frac{C}{r-\r} \left( 1 \, + \, \parallel a \parallel_{L^q(\Q_r)} \, + \, \parallel b 
    \parallel_{L^q(\Q_r)} \right) \, + \, \frac{p}{2} \parallel c \parallel_{L^q(\Q_r)} \right] \parallel v \parallel^2_{L^{2 \b}(\Q_r)} ,
    \end{align*}
    where $\b = \b(q)$ is defined in \eqref{esponenti}.
\end{proposition}

\begin{proof}
    We consider the case $p<1$, $p \ne 0$, $p \ne 1/2$. First of all, we consider an
    uniformly positive weak solution $u$ to $\L u = 0$, that is $u \ge u_0$ for some 
    constant $u_0 > 0$. For every $\psi \in C^{\infty}_0(\Q_r)$ we consider the function $\phi = u^{2p-1}\psi^2$. Note that $\phi,
    D_{m_0}\phi \in L^2 (\Q_r)$, then we can use $\phi$ as a test function in 
    \eqref{weak-sol2}:
    \begin{align*}
    	0 &=  \int_{\Q_r} \left( - \langle A Du, D(u^{2p-1}\psi^2) \rangle + u^{2p-1}\psi^2 Y u + \langle a, D (u^{2p-1}\psi^2) \rangle u +
			\langle b , Du \rangle u^{2p-1}\psi^2 + c u^{2p}\psi^2 \right)
    \end{align*}
    Let $v = u^p$. Since $u$ is a weak solution to $\L u = 0$  and $u 
    \ge u_0$, then $v, D_{m_0} v, Yv \in L^2(\Q_r)$:
    \begin{align*}
      0 = &- \int_{\Q_r} \left( 1 - \frac{1}{2p} \right) \langle A Dv , Dv \rangle  \psi^2 \, - \, \int_{\Q_r} \langle A Dv, D \psi \rangle v \psi 
      				\, + \,\frac{1}{4}  \int_{\Q_r} Y( v^2) \psi^2 \\
      	&- \, \int_{\Q_r} \div a \, v^2 \psi^2 \, - \, \frac{1}{4} \int_{\Q_r} \langle a, D(v^2) \rangle \psi^2 \, + \, \frac{1}{4} \int_{\Q_r}  \langle b, D (v^2) 
	\rangle \psi^2 \, + \, \frac{p}{2} \int_{\Q_r} c v^2 \psi^2.
    \end{align*}
    Because of assumption \textbf{(H1)} and by definition \eqref{cutoff} of the cut-off function $\psi$, we get the following inequality
    \begin{align}
        \label{estimate}
        \frac1\l \, & \left( \frac{2p - 1}{2p} + \e \right) \int_{\Q_\r} |D v|^2  \, \le \\ \nonumber
        &\le \frac{1}{4 \e \l} \frac{C}{(r-\r)^2} \int_{\Q_r} |v|^2 \, 
        \boxed{\, - 	\, \int_{\Q_r} \div a \, v^2 \psi^2 \, - \, \frac{1}{4} \int_{\Q_r} \langle a, D(v^2) \rangle \psi^2 \,}_A \, + \\ \nonumber
        &\qquad + \, \boxed{ \frac{1}{4} \int_{\Q_r} \langle b, D(v^2) \rangle \psi^2  }_B \, + \, \boxed{  \frac{p}{2} \int_{\Q_r} c v^2 \psi^2  }_C \, + \, 
        \boxed{ \frac{1}{4} \int_{\Q_r} Y(v^2) \psi^2 }_D
    \end{align}
where $\e$ is a positive constant coming from the application of the Young's inequality. In the following we are going to consider exponents $\a = \a(q)$ and 
$\b = \b(q)$ defined in \eqref{esponenti}. Now we need to estimate the boxed terms. 

Let us consider the term A, by Assumption \textbf{(H3)} and a classic H\"older estimate we have that
\begin{align*}
	\boxed{\, - \, \int_{\Q_r} \div a \, v^2 \psi^2 \, - \, \frac{1}{4} \int_{\Q_r} \langle a, D(v^2) \rangle \psi^2 \,}_A
	 \,  &\le \, - \frac{3}{4} \int_{\Q_r} \div a \, v^2 \psi^2 \, + \, \frac{1}{2} \int_{\Q_r} | \langle a, D \psi \rangle | \, |\psi| \, v^2 \\
	 &\le \, \frac{C}{r-\r} \, \parallel a \parallel_{L^{q}(\Q_r)} \parallel v \parallel_{L^{2\b}(\Q_r)}^2.
\end{align*}

Let us consider the term B. Thus, by Assumption \textbf{(H3)} and a classic H\"older estimate we have that
\begin{align*}
       \boxed{ \frac{1}{4} \int_{\Q_r} \psi^2 \langle b, D(v^2) \rangle }_B &\le \, 
       - \, \frac{1}{4} \int_{\Q_r} v^2 \psi^2 \div b \, + \, \frac{1}{2} \int_{\Q_r} |\langle b, D\psi \rangle| |\psi| v^2 \\
       &\le \, \frac{C}{r - \r} \parallel b \parallel_{L^q(\Q_r)} \, \parallel v \parallel_{L^{2\b}(\Q_r)}^2.
\end{align*}

Let us consider the linear term C. We estimate it via a classical H\"older estimate:
\begin{equation*}
	\boxed{  \frac{p}{2} \int_{\Q_r} c v^2 \psi^2  }_C \le \frac{p}{2} \, \parallel c \parallel_{L^{q}(\Q_r)} \, \parallel v \parallel_{L^{2\b}(\Q_r)}^2 .
\end{equation*}

As far as it concerns the term D, we begin considering the following equality:
\begin{equation*}
    \psi^2 Y(v^2) = Y(\psi^2 v^2) - 2 v^2 \psi Y \psi.
\end{equation*}
Since by the divergence theorem $D_1 = 0$ ($v^2 \psi^2$ is null on the boundary of $\Q_r$), we get
\begin{align*}
    \boxed{ \frac{1}{4} \int_{\Q_r} Y(v^2) \psi^2 }_D = \, D_1 \, + \, D_2 \,=  \int_{\Q_r} \frac{1}{4} Y(v^2 \psi^2) \, + \, \int_{\Q_r} \frac{v^2 \psi}{2} Y \psi \,  
    \le \, \frac{C}{\r (r-\r)} \parallel v \parallel^2_{L^{2}(\Q_r)}.
\end{align*}

Thus we have
\begin{align*}
       \frac1\l \left(  \frac{2p-1}{2p} \, + \, \e \right) & \parallel D v \parallel^2_{L^2(\Q_\r)} \, \le \, \left( \frac{c}{4 \e \l} \frac{1}{(r-\r)^2} \, + \, 
       \frac{C}{\r (r-\r)} \right) \parallel v \parallel^2_{L^2(\Q_r)} \, + \\
       & + \, \frac{C}{r-\r} \left( \parallel a \parallel_{L^q(\Q_r)} \, + \, \parallel b \parallel_{L^q(\Q_r)} \right)
       \parallel v \parallel^2_{L^{2 \b}(\Q_r)}  \, + \, \frac{p}{2} \parallel c \parallel_{L^q(\Q_r)} \parallel v \parallel^2_{L^{2 \b}(\Q_r)} .
\end{align*}
By choosing $\e = \frac{1}{2p}$ and considering that $\b > 2$ we have that
\begin{align}
    \label{estimate2}
    & \frac1\l \, \parallel D v \parallel^2_{L^2(\Q_\r)} \, \le \\ \nonumber
    &\le \left[ \frac{C \, p}{2 \l} \frac{1}{(r-\r)^2} \, + \, \frac{C}{r-\r} \left( 1 \, + \, \parallel a \parallel_{L^q(\Q_r)} \, + \, \parallel b \parallel_{L^q(\Q_r)} 
    \right) \, + \, \frac{p}{2} \parallel c \parallel_{L^q(\Q_r)} \right] \parallel v \parallel^2_{L^{2 \b}(\Q_r)} .
\end{align}

\medskip
The previous argument can be adapted to the case of a non-negative weak solution 
to $\L u= 0$. Indeed, we may consider the estimate \eqref{estimate2} for the 
solution $u + \frac{1}{n}, \hspace{1mm} n \in \N$, 
\begin{align*}
       & \frac{1}{\l} \int_{\Q_\r} \left| D \left( u + \frac{1}{n} \right)^p  
       \right|^2 \, \le \\
       &\le \, \left[ \frac{C \, p}{2 \l} \frac{1}{(r-\r)^2} \, + \, \frac{C}{r-\r} \left( 1 + \parallel a \parallel_{L^q(\Q_r)} +
       \parallel b \parallel_{L^q(\Q_r)} \right) + \frac{p}{2} \parallel c \parallel_{L^q(\Q_r)} \right] 
       \left( \int_{\Q_r}  \left( u + \frac{1}{n} \right)^{2\b} \right)^{\frac{1}{\b}}.
\end{align*}
We let $n$ go to infinity. The passage to the limit in the first integral is 
allowed because
\begin{align*}
    \left| D \left( u + \frac{1}{n} \right)^p  \right| = p \left( u + \frac{1}{n} 
    \right)^{p-1} \left| D u \right| \hspace{2mm}  \nearrow \hspace{2mm}
    \left| D u^p \right| , \hspace{4mm} \forall p < 1, \hspace{1mm} n \rightarrow \infty.
\end{align*}
For the second integral we rely on the assumptions $u^p \in L^2(\Q_r)$ and $u^p \in L^{2\frac{q}{q-1}}(\Q_r)$.

Next, we consider the case $p \ge 1$. For any $n \in \N$, we define the function 
$g_{n,p}$ on $]0, +\infty [$ as follows 
\begin{equation*}
    g_{n,p}(s) = \begin{cases}
                    s^p, &\text{if} \hspace{2mm} 0 < s \le n, \\
                    n^p + p n^{p-1} (s-n), \hspace{4mm} &\text{if} \hspace{2mm} s 
                    > n,
                 \end{cases}
\end{equation*}
then we let 
\begin{equation*}
    v_{n,p} = g_{n,p}(u).
\end{equation*}
Note that 
\begin{equation*}
    g_{n,p} \in C^1( \R^+), \hspace{4mm} g_{n,p}^{'} \in L^\infty( \R^{+}).
\end{equation*}
Thus since $u$ is a weak solution to $\L u = 0$, we have
\begin{equation*}
    v_{n,p} \in L^2_\loc, \hspace{2mm} D v_{n,p} \in L^2_\loc, \hspace{2mm} 
    Y v_{n,p} \in L^2_\loc .
\end{equation*}
We also note that the function 
\begin{equation*}
    g^{''}_{n,p} (s) = \begin{cases}
                        p (p-1) s^{p-2}, \hspace{4mm} & \text{if} \hspace{2mm} 0 < s < n \\
                        0, & \text{if} \hspace{2mm} s \ge n,
                       \end{cases}
\end{equation*} 
is the weak derivative of $g_{n,p}^{'}$, then $D g_{n,p}^{'}(u) =  g^{''}_{n,p} (u) D(u)$
(for the detailed proof of this assertion, we refer to \cite{GT}, Theorem 7.8). Hence, by
considering
\begin{equation*}
    \phi = g_{n,p}(u) \hspace{1mm} g_{n,p}^{'}(u) \hspace{1mm} \psi^2, 
    \hspace{4mm} \psi \in C^\infty_0(\Q_r)
\end{equation*}
as a test function in Definition \ref{weak-sol2}, we find 
\begin{align*}
    0 &= \int \limits_{\Q_1} - \langle A Du, D \phi \rangle + \phi Yu - \div a \, u \phi - \langle a, Du \rangle \phi + \langle b, D u \rangle \phi + cu\phi
    		\\ \nonumber
      &= \int \limits_{\Q_1} - \left( g^{'}_{n,p} (u) \right)^2 \psi^2 \langle A Du, D u \rangle 
      		- g^{''}_{n,p}(u) \, g_{n,p} (u) \psi^2 \langle A Du, D u \rangle
      				- 2 \psi \langle A Du, D \psi \rangle g_{n,p} (u) \, g^{'}_{n,p} (u)  + \\ \nonumber
	   &+ \int \limits_{\Q_1} g_{n,p} (u) \, g^{'}_{n,p} (u) \psi^2 \, Yu 
	   	- \div a \, u \, g_{n,p} (u) \, g^{'}_{n,p} (u) \psi^2 - \langle a, D u \rangle \psi^2 g_{n,p} (u) \, g^{'}_{n,p} (u) + \\ \nonumber
	   &+ \int \limits_{\Q_1} \langle b, D u \rangle \psi^2 g_{n,p} (u) \, g^{'}_{n,p} (u) + c u g_{n,p} (u) \, g^{'}_{n,p} (u) \psi^2.
\end{align*}	
Since $v=g_{n,p} (u)$ we have that the following equality holds:
\begin{align*}
    0  &= \int \limits_{\Q_r} - \psi^2 \langle A Dv_{n,p}, D v_{n,p} \rangle - \boxed{ g^{''}_{n,p}(u) \, g_{n,p} (u) \psi^2 \langle A Du, D u \rangle }_{A}
      				- 2 \psi \langle A Dv_{n,p}, D \psi \rangle v_{n,p}  + \\ \nonumber
	&+ \int \limits_{\Q_r} \frac12 \psi^2 \, Y(v_{n,p}^2) + \boxed{  \div a \left( \frac12 v_{n,p}^2 \psi^2 - u \, g_{n,p} (u) \, g^{'}_{n,p} (u) 
		\psi^2 \right) - \div b \, v_{n,p}^2 \psi^2  }_{B} \\ \nonumber
	&+ \int \limits_{\Q_r} \frac12 \langle a, D(\psi^2) \rangle v_{n,p}^2  - \langle b, D (\psi^2) \rangle v_{n,p}^2 
			+ c u g_{n,p} (u) \, g^{'}_{n,p} (u) \psi^2.
\end{align*}	
Since $g^{''}_{n,p}(u) \ge 0$ we have that the boxed term A is non-negative. Moreover, by Assumption \textbf{(H3)} the boxed term B is non-positive. Thus, by considering Assumption \textbf{(H1)} and by choosing $\e = \frac{1}{2p}$ we have that
\begin{align*}
    \frac1\l \, \int \limits_{\Q_r}  |D v_{n,p}|^2  
    &\le  \frac{C \, p}{2 \l} \frac{1}{(r-\r)^2} \int \limits_{\Q_r}  |v_{n,p}|^2 + \frac12 \langle a, D(\psi^2) \rangle v_{n,p}^2  - \langle b, D (\psi^2) \rangle 
    			v_{n,p}^2 + c u g_{n,p} (u) \, g^{'}_{n,p} (u) \psi^2 
\end{align*}
Since $0 < v_{n,p} \le u^p$ and 
\begin{equation*}
	| D v_{n,p}  | \, \uparrow \, | D u^p | , \qquad {\rm as} \, n \rightarrow \infty, 
\end{equation*}	
we get from the above inequality
\begin{align*}	
	\frac1\l \, \int \limits_{\Q_r}  |D u^p|^2  
    &\le  \frac{C \, p}{2 \l} \frac{1}{(r-\r)^2} \int \limits_{\Q_r}  |u^p|^2 + \frac12 \langle a, D(\psi^2) \rangle u^{2p}  - \langle b, D (\psi^2) \rangle 
    			u^{2p} + c u^{2p} \psi^2
\end{align*}
and we conclude the proof as in the previous case. 
\end{proof}

\section{The Moser's Iteration}
\label{moser}
In this Section we use the classical Moser's iteration scheme to prove 
Theorem \ref{iteration}. We begin with some preliminary remarks.
First of all, we recall the following Lemma, whose proof can be found in \cite{CPP}, Lemma 6.
\begin{lemma}
    \label{cylinders}
    There exists a positive constant $\overline{c} \in ]0,1[$ such that
    \begin{equation}
        \label{cylindereq}
        z \circ \Q_{\overline{c} r (r-\r)} \subseteq \Q_r,
    \end{equation}
    for every $0 < \r < r \le 1$ and  $z \in \Q_\r$. 
\end{lemma}
We are now in position to prove Theorem \ref{iteration}. 

\medskip

\noindent
{\sc Proof of Theorem \ref{iteration}.}
It suffices to give proof in the case $z_0=0$, $r \in ]0,1]$ and $0 < \r < r$. Combining Theorems
\ref{sobolev} and \ref{pcaccB1}, we obtain the following estimate: if $s, \d >0$ 
verify the condition 
    \begin{equation*}
    	|s - 1/2| \ge \d,
    \end{equation*}
    then, for every $\r, r$ such that $\frac{1}{2} \le \r < r \le 1$,  
    there exists a positive constant $\widetilde{C}$ such that
\begin{align}
	\label {it}
	\parallel u^s \parallel_{L^{2\a} (\Q_\r)} \, \le \, \widetilde{C} \left( s, \l, \parallel a \parallel_{L^{q}(\Q_{r})} , \parallel b \parallel_{L^{q}(\Q_{r})} , \parallel c \parallel_{L^{q}(\Q_{r})} \right)  
	\parallel u^s \parallel_{L^{2\b} (\Q_r)}
\end{align}
where
\begin{align*}
	\widetilde{C} & \left( s, \l, \parallel a \parallel_{L^{q}(\Q_{r})} , \parallel b \parallel_{L^{q}(\Q_{r})} , \parallel c \parallel_{L^{q}(\Q_{r})} \right) \, = \, 
	C(s, \l) \left( 1 + \parallel a \parallel_{L^{q}(\Q_{r})} + \parallel b \parallel_{L^{q}(\Q_{r})} \right) \parallel c \parallel_{L^{q}(\Q_{r})}^{\frac12} \, + \, \\ 
	&+ \, \frac{C(\l) \left( 1 + \parallel a \parallel_{L^{q}(\Q_{r})} + \parallel b \parallel_{L^{q}(\Q_{r})} \right)^{\frac32}}{(r - \r)^{\frac12}} \, + \, \frac{C}{(r-\r)^{\frac32}} \left( 1 + \parallel a 
	\parallel_{L^{q}(\Q_{r})} \, + \, \parallel b \parallel_{L^{q}(\Q_{r})} \right)^{\frac12} + \\
	&+ \, \frac{C(s)}{r-\r} \left( 1 + \parallel a \parallel_{L^{q}(\Q_{r})} + \parallel b \parallel_{L^{q}(\Q_{r})} + \l^{\frac12} \parallel c \parallel_{L^{q}(\Q_{r})}^{\frac12} \right)  
	\, + \, \frac{C(s)}{(r-\r)^{2}}.
\end{align*}

We remark that the previous constant  $\widetilde{C}$ can be estimated as follows
\begin{align}
	\label{ctilde}
	\widetilde{C} ( s, \l, \parallel a \parallel_{L^{q}(\Q_{r})}, \parallel b \parallel_{L^{q}(\Q_{r})} &, \parallel c \parallel_{L^{q}(\Q_{r})} ) \, \le \\ \nonumber
	&\le \, \frac{K(\l, s) \left( 1 + \parallel a \parallel_{L^{q}(\Q_{r})}^{2} + \parallel b \parallel_{L^{q}(\Q_{r})}^{2} + \parallel c \parallel_{L^{q}(\Q_{r})} \right)}{\left( \r_n - \r_{n+1}\right)^{2}}.
\end{align}
Fixed a suitable $\d >0$, we shall specify later on, and $p >0$ we iterate inequality \eqref{it} by choosing
\begin{equation*}
    \r_n = \r + \frac{1}{2^{n}} \left( r - \r \right), \qquad 
    p_n = \a^n \, \frac{p}{2 \b}, \hspace{6mm} n \in \N \cup \{ 0 \}. 
\end{equation*}

Then we set $v=u^{\frac{p}{2 \, \b}}$. If $p > 0$ is such that 
\begin{equation}
    \label{pcond}
    | p \, \a^n - \b| \ge 2 \b \d , \qquad \forall n \in \N \cup \{ 0 \},
\end{equation}
by \eqref{it} and estimate \eqref{ctilde} we obtain the following inequality for every $n \in \N \cup \{ 0 \}$
\begin{align}
   \label{it22}
   \parallel v^{\a^n} \parallel_{L^{2\a} (\Q_{\r_{n+1})}} \,
    \le \, \frac{ K(\l, p) \left( 1 + \parallel a \parallel_{L^{q}(\Q_{r})}^{2} + \parallel b \parallel_{L^{q}(\Q_{r})}^{2} + \parallel c \parallel_{L^{q}(\Q_{r})} \right)}{\left( \r_n - \r_{n+1}\right)^{2}} 
    \parallel v^{\a^n} \parallel_{L^{2\b} (\Q_{\r_n})}.
\end{align}
Since 
\begin{equation*}
    \parallel v^{\a^n} \parallel_{L^{2\a}} = \left( \parallel v \parallel_{L^{2\a^{n+1}}} \right)^{\a^n} \qquad {\rm and } \qquad 
    \parallel v^{\a^n} \parallel_{L^{2\b}} = \left( \parallel v \parallel_{L^{2\a^{n}}} \right)^{\a^n}
\end{equation*}
we can rewrite equation \eqref{it22} in the following form for every $n \in \N \cup \{ 0 \}$
\begin{align*}
   \parallel v \parallel_{L^{2\a^{n+1}} (\Q_{\r_{n+1})}} \le \left( 
   \frac{ K(\l, p) \left( 1 + \parallel a \parallel_{L^{q}(\Q_{r})}^{2} + \parallel b \parallel_{L^{q}(\Q_{r})}^{2} + \parallel c \parallel_{L^{q}(\Q_{r})} \right)}{\left( \r_n - \r_{n+1}\right)^{2}}  
   \right)^{\frac{1}{\a^n}} \parallel v \parallel_{L^{2\b \, \a^n} (\Q_{\r_n})}.
\end{align*}
Iterating this inequality, we obtain 
\begin{align*}
   \parallel v \parallel_{L^{2\a^{n+1}} (\Q_{\r_{n+1})}} \hspace{1mm} \le \prod_{j=0}^n  
   & \left( \frac{2^{2(j+1)}}{(r-\r)^{2}} \right)^{\frac{1}{\a^j}} \cdot \\
   & \cdot \left( K(\l, p) 
   \left( 1 + \parallel a \parallel_{L^{q}(\Q_{r})}^{2} + \parallel b \parallel_{L^{q}(\Q_{r})}^{2} + \parallel c \parallel_{L^{q}(\Q_{r})} \right)
   \right)^{\frac{1}{\a^j}} \parallel v \parallel_{L^{2\b} (\Q_{r})},
\end{align*}
and letting $n$ go to infinity, we get
\begin{equation*}
    \sup_{\Q_{\r}} v \, \le \, \frac{ \widetilde{K} }{ (r -\r)^{\m} } \parallel v \parallel_{L^{2\b} (\Q_r)} ,
\end{equation*}
where $\m = \frac{2 \a}{\a -1}$ and
\begin{equation*}
   \widetilde{K} = \prod_{j=0}^n  
   \left( K(\l, p) \left( 1 + \parallel a \parallel_{L^{q}(\Q_{r})}^{2} + \parallel b \parallel_{L^{q}(\Q_{r})}^{2} + \parallel c \parallel_{L^{q}(\Q_{r})} \right)   \right)^{\frac{1}{\a^j}}
\end{equation*}
is a finite constant dependent on $\d$. Thus, we have proved that 
\begin{equation}
    \label{moserp}
    \sup_{\Q_{\r}} u^p \, \le \, \left( \frac{ \widetilde{K} }{ (r - \r)^{\m} } \right)^{2 \b} \int_{\Q_r} u^p ,
\end{equation}
for every $p$ which verifies condition \eqref{pcond}. Because 
\begin{equation*}
	(Q+2) \, \le \, 2 \b \m \, < \, 9 (Q+2)
\end{equation*}
we get estimate \eqref{est-it}. We now make a suitable choice of $\d > 0$, only dependent on the homogeneous dimension $Q$, in order to show that \eqref{pcond} holds for every positive $p$. We remark that, if $p$ is a number of the form
\begin{equation*}
    p_m = \frac{\a^m(\a + 1)}{2 \b}, \qquad m \in \Z,
\end{equation*}
then \eqref{pcond} is satisfied with 
\begin{equation*}
    \d = \frac{|q- \frac{(Q+2)}{2}|}{(Q+2)^2}, \qquad \forall \, m \in \Z.
\end{equation*}
Therefore \eqref{moserp} holds for such a choice of $p$, with $\widetilde{K}$ only dependent on $Q, \l$ and $\parallel a \parallel_{L^q(\Q_r)}$, $\parallel b \parallel_{L^q(\Q_r)}$, $\parallel c \parallel_{L^q(\Q_r)}$. On the other hand, if $p$ is an arbitrary 
positive number, we consider $m \in \Z$ such that
\begin{equation}
    \label{p}
    p_m \le p < p_{m+1}.
\end{equation}
Hence, by \eqref{moserp} we have 
\begin{equation*}
	\sup_{\Q_{\r}} u \, \le \, \left( \frac{ \widetilde{K} }{ (r - \r)^{\m} } \right)^{\frac{2 \b}{p_{m}}} \left(  \int_{\Q_r} u^{p_{m}} \right)^{\frac{1}{p_{m}}} 
	\, \le \, \left( \frac{ \widetilde{K} }{ (r - \r)^{\m} } \right)^{\frac{2 \b}{p_{m}}} \left(  \int_{\Q_r} u^{p} \right)^{\frac{1}{p}} 
\end{equation*}
so that, by \eqref{p}, we obtain 
\begin{equation*}
    \sup_{\Q_{\r}} u^{p} \, \le \,  \left( \frac{ \widetilde{K} }{ (r - \r)^{\m} } \right)^{2 \a \b} \int_{\Q_r} u^{p} 
\end{equation*}
This concludes the proof of \eqref{est-it} for $p>0$.
We next consider $p<0$. In this case, assuming that $u \ge u_0$ for some positive 
constant $u_0$, estimate \eqref{est-it} can be proved as in the case $p>0$ or even 
more easealy since condition \eqref{pcond} is satisfied for every $p<0$. On the other
hand, if $u$ is a non-negative solution, it suffices to apply \eqref{est-it} to $u + \frac{1}{n}, \, n \in \N$, 
and let $n$ go to infinity, by the monotone convergence theorem.
\hfill $\square$

\medskip

As far as we are concerned with the proof of Corollary \ref{bounded}, it can be straightforwardly accomplished proceeding as in \cite[Corollary 1.4]{PP}.
Moreover, Proposition \ref{propsubsup} can be obtained by the same argument used in the proof of Theorem \ref{iteration}. For this reason, we do not give here 
the proof of these two results. 

\medskip

We close this Section recalling that Theorem \ref{iteration} also holds true in the sets 
\begin{equation}
	\Q^{-}_{r}((x_{0}, t_{0} )) := \Q_{r}((x_{0}, t_{0} )) \cap \{ t < t_{0} \},
\end{equation}
in the case of non-negative exponents $p$. This result is analogous to \cite{M1}, Theorem 3 (see also inequality ($6^{-}$) of Lemma 1 in \cite{M2}) and states that,
in some sense, every point of $\overline{\Q^{-}_{\r}(z_{0})}$ can be considered as an interior point of $\overline{\Q^{-}_{r} (z_{0})}$, when $\r < r$, even though it belongs to its topological boundary. 
\begin{proposition}
	Let $u$ be a non-negative weak sub-solution to $\L u = 0$ in $\O$. Let $z_{0} \in \O$ and $r, \r, \frac12 \le \r < r \le 1$, such that $\overline{\Q^{-}_{\r}(z_{0})} \subseteq \O$ and $p<0$.
	Then there exist positive constants $C= C(p, \l)$ and $\g=\g(p,q)$ such that
	\begin{equation}
		\sup_{\Q_{\r}^{-}(z_{0})} u^{p} \, \le \, \frac{C \left( 1 + \parallel a \parallel_{L^q (\Q_r(z_{0}))}^{2} + \parallel b \parallel_{L^q (\Q_r(z_{0}))}^{2} + \parallel c \parallel_{L^q (\Q_r(z_{0}))} 
		\right)^{\g}}{(r-\r)^{9 (Q+2)}} \int_{\Q_{r}^{-}(z_{0})} u^{p},
	\end{equation}
	where $\g = \frac{2 \a^{2} \b}{\a-1}$, with $\a$ and $\b$ defined in \eqref{esponenti}, provided that the integral is convergent.
\end{proposition}
The proof of the above Proposition can be straightforwardly accomplished proceeding as in Proposition 5.1 in \cite{PP}, and therefore is omitted.

\section*{Acknowledgements}
This research group was partially supported by the grant of Gruppo Nazionale per l'Analisi Matematica, la Probabilità e le loro Applicazioni (GNAMPA) of the Istituto 
Nazionale di Alta Matematica (INdAM). The second author acknowledges financial support from the FAR2017
project ``The role of Asymmetry and Kolmogorov equations in financial Risk Modelling (ARM)”. The third author is partially supported by the 
``RUDN University Program 5-100''.


\end{document}